\documentclass[a4paper]{amsart}
\usepackage{amsfonts}
\usepackage{amssymb}
\usepackage{nicefrac}
\usepackage{enumerate}
\usepackage{url,bm}
\usepackage{color}      
\usepackage{capt-of}
\usepackage{graphicx}

\newcommand{\R}{\mathbb{R}}
\newcommand{\N}{\mathbb{N}}
\newcommand{\C}{\mathbb{C}}
\newcommand{\Z}{\mathbb{Z}}
\newcommand{\esssup}{\mathop{\rm ess\,sup}}

\newcommand{\w}{\omega}

\newcommand{\s}{\bm{\mathcal{F}}}
\newcommand{\g}{\bm{\mathcal{G}}}
\newtheorem{theorem}{Theorem}[section]
\newtheorem{lemma}[theorem]{Lemma}
\newtheorem{corollary}[theorem]{Corollary}
\newtheorem{proposition}[theorem]{Proposition}

\theoremstyle{remark}
\newtheorem{remark}[theorem]{Remark}
\newtheorem{example}[theorem]{Example}

\theoremstyle{definition}
\newtheorem{definition}[theorem]{Definition}

\date{December 2013}
\thanks{The research of the authors is partially supported by Grants: CONICET PIP 398, PICT 2011-436 and
UBACyT 20020100100638.}

\begin{document}
\title{Linear combinations of frame generators in systems of translates}
\author[ C. Cabrelli, C. A. Mosquera and V. Paternostro]{C. Cabrelli, C. A. Mosquera and V. Paternostro}

\address{\textrm{(C. Cabrelli)}
Departamento de Matem\'atica,
Facultad de Ciencias Exac\-tas y Naturales,
Universidad de Buenos Aires, Ciudad Universitaria, Pabell\'on I,
1428 Buenos Aires, Argentina and
IMAS-CONICET, Consejo Nacional de Investigaciones
Cient\'ificas y T\'ecnicas, Argentina}
\email{cabrelli@dm.uba.ar}

\address{\textrm{(C. A. Mosquera)}
Departamento de Matem\'atica,
Facultad de Ciencias Exac\-tas y Naturales,
Universidad de Buenos Aires, Ciudad Universitaria, Pabell\'on I,
1428 Buenos Aires, Argentina and
IMAS-CONICET, Consejo Nacional de Investigaciones
Cient\'ificas y T\'ecnicas, Argentina}
\email{mosquera@dm.uba.ar}

\address{\textrm{(V. Paternostro)}
Departamento de Matem\'atica,
Facultad de Ciencias Exac\-tas y Naturales,
Universidad de Buenos Aires, Ciudad Universitaria, Pabell\'on I,
1428 Buenos Aires, Argentina and
IMAS-CONICET, Consejo Nacional de Investigaciones
Cient\'ificas y T\'ecnicas, Argentina}
\email{vpater@dm.uba.ar}

\begin{abstract}
A finitely generated shift invariant space $V$ is a closed subspace of  $L^2(\R^d)$ that can be generated by the integer translates of
a finite number of functions.
A set of frame generators for $V$ is a set of functions whose  integer translates form a frame for $V$.
In this note we give necessary and sufficient conditions in order that a minimal set of frame generators can be obtained by taking linear
combinations of the given frame generators.
Surprisingly the results are very different to the recently studied case when the property to be a frame is not required.
\end{abstract}

\maketitle

\section{Introduction}

 {\it Shift invariant spaces} (SISs) are  closed subspaces of $L^2(\R^d)$ that are invariant under integer translations. They play an important role in approximation theory, harmonic analysis, wavelet
theory, sampling and signal processing \cite{AG01, Gro01, HW96, Mal89}. The structure of these spaces has been deeply analyzed  (see for example \cite{dBDVR94, dBDVRI94, Bow00, Hel64, RS95}).

A set of functions $\Phi$ is a {\it set of generators} for a shift invariant space $V$ if the closure of the space spanned by all integer translations of the functions in $\Phi$
agrees with $V.$ When there exists a finite set of generators $\Phi$ for $V$, we say that $V$ is finitely generated. In this case, there exists a positive integer $\ell$,  called the length of $V$, that is defined as the minimal number of functions that generate $V.$ Any set of generators of $V$ with $\ell$ elements will be called a {\it minimal set of generators.}

Let    $\Phi=\{\phi_1,\cdots,\phi_m\}$ be a set of generators for a shift invariant space $V$. It is interesting to know whether it is possible to obtain a minimal set of generators from the given generators in $\Phi$.  There are many examples with the property that no subset of $\Phi$ is a minimal set of generators.
So, deleting elements from $\Phi$ may not be a successfull procedure.

Concerning this question, Bownik and Kaiblinger in \cite{BK06}, showed that a minimal set of generators for $V$ can be obtained from $\Phi$ by  linear combinations of its elements. Moreover, they proved that almost every set of $\ell$ functions that are linear combinations of
$\{\phi_1,\cdots,\phi_m\}$  is a minimal set of generators for  $V$ (see \cite[Theorem 1]{BK06}).
We emphasize that the linear combinations only involve the functions  $\{\phi_1,\cdots,\phi_m\}$ and not their translations.

Since linear combinations of a finite number of functions   preserve  properties such as smoothness, compact support, bandlimitedness, decay, etc,  an interesting consequence of Bownik and Kaiblinger's result is that
if the generators for $V$ have some additional property, there exists a minimal set of generators that inherits this property.

In many problems involving shift invariant spaces, it is important that the system of translates $\{T_k\phi_j\colon k\in\Z^d,\,\,j=1,\cdots, m\}$ bears a particular functional analytic structure such
as being an orthonormal basis, a Riesz basis or a frame. Therefore, it is interesting to know when a minimal
set of generators obtained by taking linear combinations of the original one  has the same structure.
More precisely, suppose that $\Phi=\{\phi_1,\cdots,\phi_m\}$ generates a shift invariant space $V$ of length $\ell$,
and assume that a new set of generators $\Psi=\{\psi_1,\cdots, \psi_{\ell}\}$ for $V$ is produced by taking linear combinations of the functions in $\Phi$. That is, assume that   $\psi_i=\sum_{j=1}^m a_{ij}\phi_j$ for $i=1,\cdots,\ell$ for some complex scalars $ a_{ij}$.
If we collect the coefficients in a matrix $A=\{a_{ij}\}_{i,j}\in \C^{\ell\times m}$, then we can write in  matrix  notation
$\Psi=A\Phi$. We would like to know, which matrices $A$ transfer the structure of $\Phi$ over $\Psi$. Precisely, we study the following question:  If we know that $\{T_k\phi_i\colon k\in\Z^d,\,\,i=1,\cdots, m\}$ is a frame for $V$, when  will $\{T_k\psi_i\colon k\in\Z^d,\,\,i=1,\cdots, \ell\}$ also be a frame for $V$?

In this paper we  answer this question completely.
As we mentioned before, the property of being  a ``set of generators" for a SIS $V$ is generically preserved
by the action of a matrix $A$ (\cite{BK06}). This is not anymore valid  for the  case of frames.
This is an  unexpected result. More than that, we were able to construct a surprising example of a  shift invariant space $V$
with a set of  generators $\Phi$ such that their integer translates
$\{T_k\phi_i\colon k\in\Z^d,\,\,i=1,\cdots,m\}$ form a frame for $V$ and with the property that  no matrix $A,$ of size $\ell\times m$ with $\ell<m,$
transform $\Phi$ into a new set of generators that form a frame for $V$.

Our main result gives exact conditions  in order that the frame property is preserved by a matrix $A$.
These conditions are in terms of a particular geometrical relation that has to be satisfied between the
nullspace of $A$ and the column space of $G_{\Phi}(\omega)$ for almost all $\omega$.
The proof uses  recent results about singular values of composition of operators
and involves the Friedrichts angle between subspaces in Hilbert spaces.
We also provide an equivalent analytic condition between $A$ and $G_{\Phi}(\omega)$ in order that this same result holds. Although we are interested in the case $\ell=\ell(V)$ (the length of the SIS $V$ under study),  most of our results are still valid when $\ell(V)\le \ell\le m.$

For completeness we include the particular  case of Riesz bases and orthonormal bases that are known.

The paper is organized as follows. In Section \ref{preliminaries} we set the definitions and results that we need. We include some
results about the eigenvalues of  conjugated matrices in  Section \ref{matrix}. Finally, in Section \ref{results} we  state and prove our main results.


\section{Preliminaries}\label{preliminaries}
We start this section by giving the basic definitions. Then, we state some known result about shift invariant spaces that we will need later.

\begin{definition}
Let $\mathcal{H}$ be a separable Hilbert  space and   $\{f_k\}_{k\in \Z}$ be a sequence in $\mathcal{H}.$

\begin{enumerate}
\item[$(a)$]
The sequence $\{f_k\}_{k\in \Z}$ is said to be a {\it Riesz basis} for $\mathcal{H}$
if it is complete in
$\mathcal{H}$ and if there exist $0<\alpha\leq \beta$ such that
 for every finite scalar sequence
$\{c_k\}_{k\in \Z}$ one has
$$\alpha \sum_{k\in\Z} |c_k|^2 \leq \|\sum_{k\in\Z} c_k f_k\|^2\leq
\beta\sum_{k\in\Z} |c_k|^2.$$
The constants $\alpha$ and $\beta$ are called {\it  Riesz bounds}.
\item[$(b)$]
The sequence $\{f_k\}_{k\in \Z}$ is said to be  a {\it frame} for $\mathcal{H}$ if there exist $0<\alpha\leq \beta$ such that
\begin{equation}\label{eq-frame}
\alpha\|f\|^2\leq \sum_{k\in \Z} |\langle f,f_k\rangle|^2\leq \beta\|f\|^2
\end{equation}
for all $f\in\mathcal{H}$. The constants $\alpha$ and $\beta $ are called {\it frame bounds}.

When only the right hand side inequality in (\ref{eq-frame}) is satisfied, we say that $\{f_k\}_{k\in \Z}$ is  {\it Bessel sequence} with Bessel bound $\beta$.
\end{enumerate}
\end{definition}

In this paper we will work with the above definitions in the following context. We will assume that $\mathcal{H}$ is a closed subspace of $L^2(\R^d)$ and the sequence $\{f_k\}_{k\in \Z}\subseteq \mathcal{H}$ consists of  integer translates of a fixed finite set of functions $\Phi\subseteq L^2(\R^d).$

\begin{definition}
We say that a closed subspace $V\subseteq  L^2(\R^d)$ is {\it shift invariant} if
$$f\in V\Longrightarrow T_kf\in V, \,\, \textrm{ for any }\,\,k\in\Z^d,$$
 where $T_k$ is the translation by the vector $k\in\Z^d,$ i.e. $T_kf(x)=f(x-k)$.

For any subset $\Phi\subseteq  L^2(\R^d)$ we define \[
S(\Phi)= \overline{\mbox{span}}\{T_k\phi\colon \phi\in\Phi, k\in\Z^d\}\,\,\textrm{ and }\,\,
E(\Phi)= \{T_k\phi\colon \phi\in\Phi, k\in\Z^d\}.
\]

We call $S(\Phi)$ the shift invariant space (SIS) generated by $\Phi$.
If $V=S(\Phi)$ for some finite set $\Phi$  we say that $V$ is a {\it finitely generated} SIS, and a
{\it principal} SIS if $V$ can be generated by the translates of a single function.
\end{definition}

For a finitely generated SIS $V\subseteq  L^2(\R^d)$ we define the length of $V$ as
\[
\ell(V)= \min\{n\in\N\colon \exists \,\,\phi_1, \cdots,\phi_n \in
V \textrm{ with } V=S(\phi_1, \cdots,\phi_n)\}.
\]
We say that $\Phi$ is a {\it minimal set of generators for $V$} if $V=S(\Phi)$ and $\Phi$ has exactly $\ell(V)$ elements.

Helson in \cite{Hel64} introduced  range functions and used this notion to completely characterize shift  invariant spaces.
Later on, several authors have used this framework to describe and characterize frames and bases of these spaces.
See for example   \cite{dBDVR94, dBDVRI94, RS95, Bow00, CP10}.
We are not going to review here the complete theory of Helson. We will only mention the required definitions and the properties that we need in this note. We refer to \cite{Bow00} for a clear and complete description.

\begin{definition}
Given $f\in L^2(\R^d)$ and $\w\in[-\frac{1}2, \frac{1}2]^d,$ the {\it fiber} $\widetilde{f_{\w}}$ of $f$ at $\w$ is the sequence
\[
\widetilde{f_{\w}} \equiv \{\widehat{f}(\w+k)\}_{k\in\Z^d}.
\]
\end{definition}
Here $\widehat{f}$ denotes the Fourier transform of the function $f,$ $\widehat{f}(\w)\!= \int_{\R^d}e^{-2\pi i\w x}f(x)\, dx$ when $f\in L^1(\R^d).$
We observe that if $f\in L^2(\R^d),$ then the fiber $\widetilde{f_{\w}}$ belongs to $\ell^2(\Z^d)$ for almost every $\w\in[-\frac{1}2, \frac{1}2]^d.$

Let $\Phi= \{\phi_1,\cdots,\phi_m\}$ be a finite collection of functions in $L^2(\R^d)$.
The \emph{Gramian} $G_\Phi$ of $\Phi$ is the
$m\times m$ matrix of $\Z^d$-periodic functions
\begin{equation} \label{gram}
[G_{\Phi}(\omega)]_{ij}
=
\sum_{k \in \Z^d} \widehat{\phi}_i(\omega+k) \, \overline{\widehat{\phi}_j(\omega+k)}.
\end{equation}
The Gramian of $\Phi$ is determined a.e. by its values at $\omega\in [-\frac{1}2, \frac{1}2]^d$ and satisfies $G_{\Phi}(\omega)^*=G_{\Phi}(\omega)$ for a.e. $\omega\in [-\frac{1}2, \frac{1}2]^d$. We denote the set of eigenvalues of $G_{\Phi}(\omega)$ by $\Sigma(G_{\Phi}(\omega))$.

For  a finitely generated SIS $V$,  the length of $V$ can be expressed in terms of the Gramian as follows (see \cite{Bow00, dBDVR94, TW12})
\begin{equation}\label{length-gramian}
\ell(V)=\esssup_{\w\in [-\frac1{2}, \frac1{2}]^d}\text{rk}(G_{\Phi}(\w))
\end{equation}
where $\text{rk}(B)$ denotes the rank of a matrix $B$ and $\Phi$ is a generator set for $V$.

For $B_1, B_2\in \C^{m\times m}$, we write $B_1\leq B_2$ meaning that $B_2-B_1$ is a positive semidefinite matrix. Using the Gramian matrix, the following characterizations hold (see \cite{Bow00}).
\begin{proposition}\label{gramiano-frame}
Let $\Phi$ a finite set of functions in $L^2(\R^d)$ and $V=S(\Phi)$. Then,
\begin{enumerate}
\item the following statements are equivalent:
\begin{enumerate}
\item $E(\Phi)$ is a Bessel sequence with bound $0<\beta$.
\item $\esssup_{\w\in [-\frac1{2}, \frac1{2}]^d} \|G_{\Phi}(\w)\|\le \beta.$
\end{enumerate}

\item the following statements are equivalent:
\begin{enumerate}
\item $E(\Phi)$ is a Riesz basis for $V$ with bounds $0<\alpha\leq \beta$.
\item For a.e. $\w\in [-\frac1{2}, \frac1{2}]^d$,
$\alpha I\leq G_{\Phi}(\w)\leq I\beta .$
\item For a.e. $\w\in [-\frac1{2}, \frac1{2}]^d$,  $\Sigma(G_{\Phi}(\w))\subseteq [\alpha, \beta]$.
\item For a.e. $\w\in [-\frac1{2}, \frac1{2}]^d,$ $G_{\Phi}(\w)$ is invertible, $\|G_{\Phi}(\w)\|\le \beta$ and $\|(G_{\Phi}(\w))^{-1}\|\le \frac{1}{\alpha}.$
\end{enumerate}
\item the following statements are equivalent:
\begin{enumerate}
\item $E(\Phi)$ is a frame for $V$ with bounds $0<\alpha\leq \beta$.
\item For a.e. $\w\in [-\frac1{2}, \frac1{2}]^d$,
$\alpha G_{\Phi}(\w)\leq G_{\Phi}^2(\w)\leq \beta G_{\Phi}(\w).$
\item For a.e. $\w\in [-\frac1{2}, \frac1{2}]^d$,  $\Sigma(G_{\Phi}(\w))\subseteq [\alpha, \beta]\cup\{0\}$.
\end{enumerate}
\end{enumerate}
\end{proposition}

As we  mentioned above, we are interested  in when a set of  linear combinations of the generators for a finitely generated SIS inherits some particular structure from the original generators. In order to make clear our exposition we use the following notation.

Let  $\Phi=\{\phi_1,\cdots,\phi_m\}$ be a set of functions in $L^2(\R^d)$.
By taking linear combinations of the elements of $\Phi$, we construct a new set $\Psi=\{\psi_1,\cdots, \psi_{\ell}\}$
where $\psi_i=\sum_{j=1}^ma_{ij}\phi_j$ and
$A=\{a_{ij}\}_{i,j}\in \C^{\ell\times m}$ is a matrix.
Using a matrix notation we set $\Psi=A\Phi$.
Then, we consider the following questions:

Let $V=S(\Phi)$ where $\Phi=\{\phi_1,\cdots,\phi_m\}$.
\begin{itemize}
\item If $E(\Phi)$ is a frame
for $V$ and $\ell(V)\leq \ell\leq m$, for which matrices $A\in \C^{\ell\times m}$, is $E(A\Phi)$ a frame
for $V$?
\item If $E(\Phi)$ is an orthonormal basis (Riesz basis)
for $V$, for which square matrices $A$, is $E(A\Phi)$ an orthonormal basis (Riesz basis)
for $V$?

\end{itemize}

Sometimes in the paper  by convenient abuse of notation,
we will say that a set $\Phi=\{\phi_1,\cdots, \phi_m\}$ is a frame for a SIS $V$ to indicate that actually $E(\Phi)$ forms a frame for $V$.

In order to study when $E(A\Phi)$ is a frame or a Riesz basis for $V$ we will use Proposition \ref{gramiano-frame}.
So, we need to know   the Gramian associated to
$A\Phi$.

\begin{proposition}\label{gramiano-A}
Let $V=S(\Phi)$ be a SIS where $\Phi=\{\phi_1,\cdots,\phi_m\}$  and let
$A=\{a_{ij}\}_{i,j}\in \C^{\ell\times m}$ be a matrix. Consider the set
$\Psi=\{\psi_1,\cdots, \psi_{\ell}\}$ where $\Psi=A\Phi$.
Then, the Gramian  $G_{\Psi}$ is a conjugation of $G_{\Phi}$ by $A,$ i.e.
$$G_{\Psi}(\w)=AG_{\Phi}(\w)A^*$$
for a.e. $\w\in [-\frac1{2}, \frac1{2}]^d.$
\end{proposition}

\begin{proof}
Let $i,j\in\{1,\cdots,\ell\}$ be fixed. Then,
\begin{align*}
[G_{\Psi}(\omega)]_{ij}
=
\Big\langle (\widetilde{\psi_i)}_{\omega}, (\widetilde{\psi_j})_{\omega}\Big\rangle
&=
\Big\langle \sum_{k=1}^ma_{ik}(\widetilde{\phi_k})_{\omega}, \sum_{r=1}^ma_{ir}(\widetilde{\phi_r})_{\omega}\Big\rangle\\
=
\sum_{k=1}^m\sum_{r=1}^ma_{ik}\overline{a_{ir}}\Big\langle (\widetilde{\phi_k})_{\omega}, (\widetilde{\phi_r})_{\omega}\Big\rangle
&=\sum_{k=1}^m\sum_{r=1}^ma_{ik}\overline{a_{ir}}[G_{\Phi}(\omega)]_{kr}
=
[AG_{\Psi}(\omega)A^*]_{ij}.
\end{align*}
\end{proof}
We study in the following section the eigenvalues of a conjugated matrix,
and provide the definition of the Friedrichs angle between subspaces, that we will need to state the main results.


\section{Eigenvalues of a conjugated matrix}\label{matrix}

According to Proposition \ref{gramiano-frame}, we will need to study the eigenvalues of the Gramian $G_{\Psi}(\omega)$
which is, as Proposition \ref{gramiano-A} shows, a conjugation of $G_{\Phi}(\omega)$ by $A$.
 The behaviour of the eigenvalues of the conjugation of a given matrix is in general
not very well established. In our case, we shall to find uniform bounds for the eigenvalues of
$G_{\Psi}(\omega)$.

In this section we first set some matrix notation  and then we state the known results
that we will later apply to shift invariant spaces.

For a matrix $B\in\C^{\ell\times m}$
we denote by $\sigma(B)$ the smallest non-zero singular value of $B$. By $Ker(B)$ and $Im(B)$ we denote the nullspace and column space
of $B$ respectively, as an operator acting by right multiplication, i.e., matrix-vector multiplication.

For a squared positive-semidefinite matrix  $B$ such that $B=B^*$,  the eigenvalues and its singular values agree. In particular, $\sigma(B)=\lambda_-(B)$ where
$\lambda_-(B)$ denotes the smallest non-zero eigenvalue  of $B$.

We now state a recent result that we need for next section. It was proven by Antezana et al. in \cite{ACRS05}.
For this, we need the notion of Friedrichs angle between subspaces.
 The Friedrichs angle can be defined for subspaces of a general Hilbert space (see \cite{Deu95, HJ95, Ka84}).
However, we will define it  for subspaces of $\C^n$, since this is the context
on which we will use it.

Let $N, M\neq \{0\}$ be subspaces of $\C^n$. The {\it Friedrichs  angle between $M$ and $N$}
is the angle in $[0, \frac{\pi}{2}]$ whose cosine is defined by
$$
{\bm{\mathcal{G}}}[M,N]=\sup\{|\langle x,y\rangle|:\, x\in M\cap (M\cap N)^{\perp},\, \|x\|=1, \, y\in N\cap (M\cap N)^{\perp},\, \|y\|=1\}.
$$
We define ${\bm{\mathcal{G}}}[M,N]=0$ if $M=\{0\}$, $N=\{0\}$, $M\subseteq N$ or $N\subseteq M$.

As usual, the sine of the Friedrichs angle is defined as $\s\,[M,N]=\sqrt{1-{\bm{\mathcal{G}}}[M,N]^2}.$
It satisfies $\s[M,N]=\s[N,M]=\s[M^{\perp},N^{\perp}]$ (see \cite{ACRS05}).

The following theorem is a particular case of the result stated in Remark 2.10 in \cite{ACRS05} for  square matrices.

\begin{theorem}\label{prop-jorge}
Let $A, B$ be non zero square matrices in $\C^{m\times m}$. Then,
$$\sigma(A)\sigma(B) \,\s[Ker(A), Im(B)]\leq \sigma(AB)\le \|A\|\|B\| \,\s [Ker(A), Im(B)].$$
\end{theorem}

In order to adapt Theorem \ref{prop-jorge} to our setting we need the following lemma.

\begin{lemma}\label{lema-rango-2}
Let $A\in \C^{\ell\times m}$ and $G\in \C^{m\times m}$ with $\ell\le m$.
If $\text{rk}(AGA^*)=\text{rk}(G)$ then $Ker(AG)=Ker(G)$.
\end{lemma}

\begin{proof}
Clearly $Ker(G)\subseteq Ker(AG)$. Suppose that $\dim Ker(G)<\dim Ker(AG)$. Then $\text{rk}(AG)<\text{rk}(G)$. Now,
$\text{rk}(AGA^*)\leq \min\{\text{rk}(AG), \text{rk}(A)\}\leq \text{rk}(AG)<\text{rk}(G)$ which is a contradiction. Thus, we must have $\dim Ker(G)=\dim Ker(AG)$.
\end{proof}
Now, from Theorem \ref{prop-jorge} and Lemma \ref{lema-rango-2} we obtain:
\begin{proposition}\label{prop-jorge-rectangular}
Let $G$ be a  positive-semidefinite matrix in $\C^{m\times m}$ such  that  $G=G^*$ and $A=\{a_{ij}\}_{i,j}\in\C^{\ell\times m}$  with $\ell\le m$ such that $\text{rk}(G)=\text{rk}(AGA^*)$.
Then,
$$\sigma(A)^2\lambda_-(G) \s[Ker(A), Im(G)]^2\leq \lambda_-(AGA^*)\le \|A\|^2\|G\| \s[Ker(A), Im(G)].$$
\end{proposition}

\begin{proof}
Let $\widetilde A\in \C^{m\times m}$ be the matrix defined by
$\widetilde{a}_{ij}=
\begin{cases}
a_{ij}& \textrm{ if } 1\leq i\leq \ell\\
0&  \textrm{ if } \ell< i\leq m\\
\end{cases}$. Then,
$$\widetilde{A}G(\widetilde{A})^*=
\Big(
\begin{array}{c|c}
AGA^* & 0\\ \hline
0 & 0
\end{array}
\Big),
$$
and therefore $ \lambda_-(AGA^*)=\lambda_-(\widetilde{A}G\widetilde{A}^*)$.

Now, using Theorem \ref{prop-jorge} and Lemma \ref{lema-rango-2}, we have
\begin{align*}
\lambda_-(\widetilde{A}G\widetilde{A}^*)&\geq
\sigma(\widetilde{A}G)\sigma(\widetilde{A}^*)\s[Ker(\widetilde{A}G), Im(\widetilde{A}^*)]\\
&=\sigma(\widetilde{A}G)\sigma(\widetilde{A}^*)\s[Ker(G), Im(\widetilde{A}^*)]\\
&\geq \sigma(\widetilde{A})\lambda_-(G) \s[Ker(\widetilde{A}), Im(G)]\sigma(\widetilde{A}^*)\s[Ker(G), Im(\widetilde{A}^*)],
\end{align*}
and
$$
\lambda_-(\widetilde{A}G\widetilde{A}^*)\leq
\|\widetilde{A}G\|\|\widetilde{A}^*\| \s[Ker(\widetilde{A}G), Im(\widetilde{A}^*)]
\le \|\widetilde{A}\|^2\|G\| \s[Ker(G), Im(\widetilde{A}^*)].
$$

By the properties of the sine of the Friedrichs angle it can be seen that $\s[Ker(G), Im(\widetilde{A}^*)]= \s[Ker(\widetilde{A}), Im(G)].$
Using that $\sigma(A)=\sigma(\widetilde{A})=\sigma(\widetilde{A}^*)$  and $\|\widetilde{A}\|=\|A\|,$ we finally obtain
$$
\sigma(A)^2\lambda_-(G)\s[Ker(\widetilde{A}), Im(G)]^2\le \lambda_-(\widetilde{A}G\widetilde{A}^*)\le
\|A\|^2\|G\| \s[Ker(\widetilde{A}), Im(G)].
$$

We finish the proof by observing that $Ker(A)=Ker(\widetilde{A})$.
\end{proof}

The last result of this section gives an equivalent condition for $\text{rk}(AGA^*)=\text{rk}(G)$  to hold, and we
will use it in the next section.

\begin{lemma}\label{lemma-Radu}
Let $A\in \C^{\ell\times m}$ with $\ell\le m$ and $G\in \C^{m\times m}$ such that $G$ is positive-semidefinite and $G=G^*$.
Then, $\text{rk}(AGA^*)=\text{rk}(G)$ if and only if $Ker(A)\cap Im(G)=\{0\}$.
\end{lemma}

\begin{proof}
Note that, since $G$ is positive semidefinite, and $G=G^*,$ it is always true that  $\text{rk}(AGA^*)=\text{rk}(AG^{1/2}G^{1/2}A^*)= \text{rk}(AG^{1/2})$. Then, $\text{rk}(AGA^*)=\text{rk}(G)$ if and only if
$\dim(Ker(AG^{1/2}))=\dim(Ker(G))$. Thus, since $Im(G)=Im(G^{1/2})$ and then $Ker(G)=Ker(G^{1/2})$, we want to prove that
$\dim(Ker(AG^{1/2}))=\dim(Ker(G^{1/2})$ if and only if $Ker(A)\cap Im(G^{1/2})=\{0\}$.
Now, using that $Ker(G^{1/2})\subseteq Ker(AG^{1/2})$, we have $\dim(Ker(AG^{1/2}))=\dim(Ker(G^{1/2})$ if and only if
$Ker(AG^{1/2})=Ker(G^{1/2})$.

Finally, it is easy to check that the condition $Ker(AG^{1/2})=Ker(G^{1/2})$ is equivalent to
$Ker(A)\cap Im(G)=\{0\}$.
\end{proof}


\section{Main results}\label{results}

As we have mentioned in the introduction, there are examples of  sets of generators  that do not
contain any {\it minimal} subset of generators. For instance,
consider the shift invariant space $V=S(\phi)$ where $\phi$ is such that
$ \widehat{\phi}=\chi_{[0,1]}$. (Here  we denote by $\chi_M$ the characteristic
function of a set $M$). It can be seen that $V=S(\phi_1, \phi_2)$ with $\phi_1, \phi_2$
such that $\widehat{\phi_1}=\chi_ {[0,\frac1{2}]}$ and $\widehat {\phi_2}=\chi_ {[\frac1{2}, 1]}$.
However, neither $\phi_1$ nor $\phi_2$ generates $V$ by itself.

An alternative to overcome this problem will be  to try to obtain  a smaller set of generators, for example a minimal one,  considering  instead  linear
combinations of the original generators. Related to this,  Bownik and
Kaiblinger (\cite{BK06}) proved the following:

\begin{theorem}\cite[Theorem 1]{BK06}\label{thm-bownik}
 Let $V$ be a finitely generated SIS  with length $\ell(V)$, and $\Phi=\{\phi_1,\cdots, \phi_m\}\subseteq L^2(\R^d)$ where  $\ell(V)\leq m$ and such that $V=S(\Phi)$.
For every $\ell(V)\leq \ell\leq m$, consider the set of matrices
$\mathcal{R}=\{A\in\C^{\ell\times m}\,:\, V=S(\Psi), \Psi=A\Phi \}.$ Then,
$\C^{\ell\times m}\setminus \mathcal{R}$ has zero Lebesgue measure.
\end{theorem}
This result briefly says that given any set of generators of a finitely generated SIS, almost every matrix (of the right size) transforms it in a new set of
generators. In particular, a {\it minimal } set of generators can be obtained with this procedure.

An interesting question arises here. When is the set of generators obtained, a set of {\it frame} generators? (i.e. the integer translations form a frame of the SIS?)
More precisely we want to obtain new sets of generators of the form $\Psi=A\Phi$ with the additional property of being a frame.

First, we show that in order to  $E(\Psi)$ to  be a frame, $E(\Phi)$ needs  to be a frame. More precisely we prove:

\begin{proposition}\label{frametoframe}
Let $ \Phi=\{\phi_1,\cdots, \phi_m\}\subseteq L^2(\R^d)$ be a generator set for a SIS $V$ of length $\ell(V)\leq m$  and suppose that $E(\Phi)$ is a Bessel sequence but not a frame for $V$. Then, for each matrix  $A\in\C^{\ell\times m}$, with $\ell(V)\leq \ell\leq m$, $E(\Psi)$ is not a frame for $V$ where $\Psi=A\Phi$.
\end{proposition}

After this result,  the right question will be which matrices (of the right size) map frame generator sets into {\it new} frame generator sets?
The answer of this question is not as direct as is the case of a plain set of generators and it will take us the rest of the section.
Let us first start with the proof of Proposition \ref{frametoframe}

\begin{proof}[Proof of Proposition\ \ref{frametoframe}]
Let $A\in\C^{\ell\times m}$. If $\Psi=A\Phi$ is not a generator set for $V$, then $\Psi$ it is not a frame for $V$.
Thus, suppose that $\Psi$ is a set of generators for $V$.

Since $E(\Phi)$ is a Bessel sequence but not a frame for $V$, the lower frame inequality in (\ref{eq-frame}) is not satisfied. Therefore, there exists $\{f_n\}_{n\in\N}\subseteq V$ such that
$$t_n(\Phi):=\sum_{j=1}^{m}\sum_{k\in\Z^d}|
\langle f_n, T_k\phi_j\rangle|^2\rightarrow 0, \textrm{ when } n\to+\infty.$$
Now,
\begin{align*}
t_n(\Psi)&=\sum_{i=1}^{\ell}\sum_{k\in\Z^d}|
\langle f_n, T_k\psi_i\rangle|^2
=\sum_{i=1}^{\ell}\sum_{k\in\Z^d}|\sum_{j=1}^{m}\overline{a_{ij}}
\langle f_n, T_k\phi_j\rangle|^2\\
&\leq \sum_{i=1}^{\ell}\sum_{k\in\Z^d}
\left(\sum_{j=1}^{m}|a_{ij}|^2\right)\left(\sum_{j=1}^{m}
|\langle f_n, T_k\phi_j\rangle|^2\right)\\
&=\left(\sum_{i=1}^{\ell}\sum_{j=1}^{m}|a_{ij}|^2\right) t_n(\Phi).
\end{align*}
Then, $t_n(\Psi)\rightarrow 0, \textrm{ when } n\to+\infty$ and thus, $E(\Psi)$ does not satisfy the lower frame inequality.
\end{proof}

When $E(\Phi)$ is not a Bessel sequence it can happen that for certain matrix $A,$ $E(A\Phi)$ is a Bessel sequence. To construct
an easy example take $\phi\in L^2(\R)$ such that $E(\phi)$ is a frame for $S(\phi).$ Now, choose a second generator $\widetilde{\phi}\in S(\phi)$
such that  $E(\widetilde{\phi})$ is not a Bessel sequence. Thus, if $\Phi= \{\phi, \widetilde{\phi}\},$ $S(\Phi)= S(\phi)$ and $E(\{\phi, \widetilde{\phi}\})$ is
not a Bessel sequence. However, taking $A= (1, 0),$ we get that $E(A\Phi)$ is a Bessel sequence. The hypothesis in the above proposition of $E(\Phi)$ being a Bessel sequence it is not very restrictive and greatly simplifies the treatment.

We now point out a property about the result of Bownik and  Kaiblinger that will be important in what follows.

From the proof of Theorem \ref{thm-bownik}, it follows that the set $\mathcal{R}$ can be described in terms of the Gramian as
\begin{equation}\label{pinta-R}
\mathcal{R}=\{A\in\C^{\ell\times m}\,:\, \text{rk}(G_{\Phi}(\w))=\text{rk}(AG_{\Phi}(\w)A^*)
\textrm{  for a.e. } \w\in [-\frac1{2}, \frac1{2}]^d   \},
\end{equation}
where $\ell$ is a number between  length of $S(\Phi)$ and $m.$ Now, using Lemma \ref{lemma-Radu}
$$\mathcal{R}=\{A\in\C^{\ell\times m}\,:\, Ker(A)\cap Im(G_{\Phi}(\w))=\{0\}
\textrm{  for a.e. } \w\in [-\frac1{2}, \frac1{2}]^d   \}.$$

\begin{remark}
 When $\ell$ is exactly the length of $S(\Phi)$, note that if $A\in\mathcal{R}$, by $(\ref{length-gramian})$, $\ell=\text{rk}(A)$ and then $A$ is full rank.
\end{remark}

As we have already discussed, a key point  in our problem is the behavior of the eigenvalues of conjugated matrices.
Here one wants to get a smaller
set of generators from a given large set of generators. In terms of matrices, this translates
in conjugating the Gramian by rectangular matrices.
The behavior of the eigenvalues in this case is not very well understood.
However we are able to  exactly determine those matrices that yield frames, in terms of the Friedrichs angle,
and using recent results by Antezana et al.\cite{ACRS05} on singular values of composition of operators.

The main result of this section is the following theorem:

\begin{theorem}\label{thm-frames}
Let $ \Phi=\{\phi_1,\cdots, \phi_m\}\subseteq L^2(\R^d)$ such that $E(\Phi)$ is a frame for $V=S(\Phi)$ and suppose that $\ell(V)\leq\ell\leq m$.
Let $A\in\C^{\ell\times m}$ be a matrix and consider $\Psi=\{\psi_1,\cdots, \psi_{\ell}\}$ where $\Psi=A\Phi$. Then, $E(\Psi)$ is a
frame for $V$ if and only if $A$ satisfies the following two conditions
\begin{enumerate}
\item $A\in\mathcal{R}$ where $\mathcal{R}$ is as in (\ref{pinta-R}).
\item There exists $\delta>0$ such that $\s\,[Ker(A), Im(G_{\Phi}(\w))]\ge \delta$ for a.e. $\w\in[-\frac1{2}, \frac1{2}]^d.$
\end{enumerate}
\end{theorem}

\begin{proof}[Proof of Theorem\ \ref{thm-frames}]
Let $0<\alpha\leq\beta$ be the frame bounds for $E(\Phi)$.

First suppose that $E(\Psi)$ is a frame for $V$ and let $\beta'\ge \alpha'>0$ be its frame bounds. Since, in particular, $\Psi$ is a generator
set for $V,$ $A$ belongs to $\mathcal{R}.$

By Proposition \ref{prop-jorge-rectangular}, we have that
$$
\lambda_{-}(G_{\Psi}(\w))= \lambda_{-}(AG_{\Phi}(\w)A^*)\le \|A\|^2\|G_{\Phi}(\w)\| \,\s[Ker(A), Im(G_{\Phi}(\w))],
$$
for a.e. $\w\in[-\frac1{2}, \frac1{2}]^d.$

Using Proposition \ref{gramiano-frame},
$\alpha'\le\lambda_{-}(G_{\Psi}(\w))$ and $\|G_{\Phi}(\w)\|\leq \beta$. Thus,
$$
\alpha'\le \|A\|^2\beta \, \s[Ker(A), Im(G_{\Phi}(\w))].
$$
Then, item (2) is satisfied taking $\delta= \frac{\alpha'}{\|A\|^2\beta}.$

Conversely. Since $A\in \mathcal{R},$ $V=S(\Psi)$ and $\text{rk}(AG_{\Phi}(\w)A^*)= \text{rk}(G_{\Phi}(\w))$ for a.e. $\w\in[-\frac1{2}, \frac1{2}]^d.$

Now, for a.e. $\w\in[-\frac1{2}, \frac1{2}]^d,$ we apply the lower inequality that Proposition \ref{prop-jorge-rectangular} gives to $G_{\Phi}(\w)$ and $A.$ Then,
$$
\lambda_{-}(AG_{\Phi}(\w)A^*)\ge \sigma(A)^2\lambda_{-}(G_{\Phi}(\w))\,\s[Ker(A), Im(G_{\Phi}(\w))]^2\ge \sigma(A)^2 \alpha \delta^2.
$$

On the other hand, by Proposition \ref{gramiano-frame},
$$
\|G_{\Psi}(\w)\|= \|AG_{\Phi}(\w)A^*\|\le \|A\|^2\|G_{\Phi}(\w)\|\le \|A\|^2\beta
$$
and from this it follows that the eigenvalues of $G_{\Psi}(\w)$ are bounded above by $\|A\|^2\beta.$

Therefore, $\Sigma(G_{\Psi}(\w))\subseteq \big[ \,\sigma(A)^2 \alpha \delta^2, \|A\|^2\beta\,\big] \cup\{0\}$ for a.e. $\w\in[-\frac1{2}, \frac1{2}]^d$
and the result follows from Proposition \ref{gramiano-frame}.
\end{proof}

As a consequence of the above theorem, we have the following result. We impose more restrictive condition on $G_{\Phi}(\w)$ than in Theorem  \ref{thm-frames}.
However the new hypothesis is easy  to check and  avoid the calculation of the sine of the Friedrichs angle.

\begin{corollary}\label{thm-frames-2}
Let $ \Phi=\{\phi_1,\cdots, \phi_m\}\subseteq L^2(\R^d)$ such that $E(\Phi)$ is a frame for $V=S(\Phi)$ and suppose that $\ell(V)\leq\ell\leq m$. Consider $A\in\C^{\ell\times m}$ such that $Ker(A)=
Ker(G_{\Phi}(\w))$ for a.e. $\w\in[-\frac1{2}, \frac1{2}]^d$, and $dim(Ker(A)) = m-\ell$.
If $A\in\mathcal{R}$ where $\mathcal{R}$ is as in (\ref{pinta-R}), and  $\Psi=\{\psi_1,\cdots, \psi_{\ell}\}$ where $\Psi=A\Phi$, then $E(\Psi)$ is a frame for $V$.
\end{corollary}

\begin{proof}
Since $Ker(A)\subseteq Ker(G_{\Phi}(\w))$ and $Ker(G_{\Phi}(\w))=Im(G_{\Phi}(\w))^{\perp}$, it follows that  $\s[Ker(A), Im(G_{\Phi}(\w))]=1$.
\end{proof}

\begin{example}
Let $V= S(\Phi)$ be the shift invariant space generated by $\Phi= \{\phi_1, \phi_2,\phi_3\}\subseteq L^2(\R),$ where  $\phi_1, \phi_2,\phi_3$ are defined by
\begin{align*}
\widehat{\phi_1}(\w)&= -8\chi_{[-\frac1{2}, \frac1{2}]}(\w)+ 4\chi_{[\frac1{2}, \frac3{2}]}(\w),\\
\widehat{\phi_2}(\w)&= \chi_{[-\frac1{2}, \frac1{2}]}(\w)+ 4\chi_{[\frac3{2}, \frac5{2}]}(\w),\\
\widehat{\phi_3}(\w)&= \chi_{[\frac1{2}, \frac3{2}]}(\w)+ 8\chi_{[\frac3{2}, \frac5{2}]}(\w).
\end{align*}
Then, the associated Gramian is
$$G_{\Phi}(\w)=
\left(
\begin{array}{ccc}
 80& -8& 4\\
 -8& 17& 32\\
 4& 32& 65\\
\end{array}
\right).
$$
It can be seen that $G_{\Phi}^2(\w)= 81G_{\Phi}(\w) $ and $\text{rk}(G_{\Phi}(w))=2$ for all $\w\in[-\frac1{2}, \frac1{2}].$
Thus, by Proposition \ref{gramiano-frame}, $E(\Phi)$ is a frame for $V$ and using \eqref{length-gramian}, $\ell(V)=2.$

On the other side, $Ker(G_{\Phi}(\w))= \text{span} \{(1,8,-4)\}.$ Then, $V$  satisfy the hypothesis of Corollary \ref{thm-frames-2}. In what follows,
we construct  all possible matrices $A\in\C^{2\times3}$  such that  $Ker(A)=  \text{span} \{(1,8,-4)\}$ and $A\in\mathcal{R}$.
These matrices  give frames $E(\Psi)$ with $\Psi=A\Phi.$

Since $Ker(A)=  \text{span} \{(1,8,-4)\},$ $A$ has the form
$$
A= \left(
\begin{array}{ccc}
 (4b-8a)& a& b\\
 (4d-8c)& c& d\\
\end{array}
\right),
$$
with $a,b,c,d\in\C.$

Now, $A\in\mathcal{R}$ if and only if $\text{rk}(AG_{\Phi}(\w)A^*)= 2.$ Now $\text{rk}(AG_{\Phi}(\w)A^*)= 2$ if and only if $\det(AG_{\Phi}(\w)A^*)\neq 0.$
Since $\det(AG_{\Phi}(\w)A^*)= (81)^3(ad-bc)^2,$ we conclude that $A\in\mathcal{R}$ if and only if $ad-bc\neq 0.$
\hfill
$\blacksquare$
\end{example}

Condition $(2)$ in Theorem \ref{thm-frames} is a geometric property that $A$ and $G_{\Phi}$
need to satisfy in order to $A$ preserves the frame property of $E(\Phi)$ over $E(\Psi)$. We now state a result
on which we give an analytic way to express
conditions $(1)$ and $(2)$ of Theorem \ref{thm-frames}.
In this case, $\ell$ will be exactly the length of the SIS.

\begin{theorem}\label{thm-Radu}
Let $ \Phi=\{\phi_1,\cdots, \phi_m\}\subseteq L^2(\R^d)$ such that $E(\Phi)$ is a frame for $V=S(\Phi)$ and suppose that $\ell(V)=\ell\leq m$.
Let $A\in\C^{\ell\times m}$ be a matrix and consider $\Psi=\{\psi_1,\cdots, \psi_{\ell}\}$ where $\Psi=A\Phi$.
Then, $E(\Psi)$ is a
frame for $V$ if and only if  the following  condition between $A$ and $G_{\Phi}$ is satisfied: $AA^*$ is invertible and
\begin{equation}\label{condition-Radu}
 \esssup_{\w\in [-\frac1{2}, \frac1{2}]^d} \|(I_m-A^*(AA^*)^{-1}A)G_{\Phi}(\w)G_{\Phi}^{\dagger}(\w)\|<1.
\end{equation}
Here, $I_m$ is the identity in $\C^{m\times m}$ and $G_{\Phi}^{\dagger}(\w)$ is the Moore-Penrose pseudoinverse of
$G_{\Phi}(\w)$.
\end{theorem}

\begin{proof}
Let us first prove that conditions $(1)$ and $(2)$ of Theorem \ref{thm-frames}
imply condition (\ref{condition-Radu}).

Note that condition $(2)$ is equivalent to $\g [Ker(A), Im(G_{\Phi}(\w))]\leq \sqrt{1-\delta^2}$ for
a.e. $\w\in [-\frac1{2}, \frac1{2}]^d$. Now, since $Ker(A)\cap Im(G_{\Phi}(\w))=\{0\}$ for a.e
$\w\in [-\frac1{2}, \frac1{2}]^d$, it follows from Proposition 2.2 in \cite{ACRS05} that
$\g [Ker(A), Im(G_{\Phi}(\w))]=\|P_{Ker(A)}P_{Im(G_{\Phi}(\w))}\|$ for a.e
$\w\in [-\frac1{2}, \frac1{2}]^d$, where $P_{Ker(A)}$ and $P_{Im(G_{\Phi}(\w))}$ denote the orthogonal projection onto
$Ker(A)$ and $Im(G_{\Phi}(\w)$ respectively.
Using the Moore-Penrose pseudoinverse we can write
$\|P_{Ker(A)}P_{Im(G_{\Phi}(\w))}\|=\|(I_m-A^{\dagger}A)G_{\Phi}(\w)G_{\Phi}^{\dagger}(\w)\|$.
Finally, since $A$ is full rank, we replace $A^{\dagger}=A^*(AA^*)^{-1}$ and then the assertion follows.

Conversely. Suppose that (\ref{condition-Radu}) holds. Then,
we have that $\|P_{Ker(A)}P_{Im(G_{\Phi}(\w))}\|\leq\gamma<1$ for all
$\w\in [-\frac1{2}, \frac1{2}]^d\setminus Z$ where $Z$ is a set with Lebesgue measure zero and
$\gamma=\esssup_{\w\in [-\frac1{2}, \frac1{2}]^d} \|(I_m-A^*(AA^*)^{-1}A)G_{\Phi}(\w)G_{\Phi}^{\dagger}(\w)\|$.

Fix $\w\in[-\frac1{2}, \frac1{2}]^d\setminus Z$ and suppose there exists $x\in Ker(A)\cap Im(G_{\Phi}(\w))$ with $\|x\|=1$.
Then, $\|P_{Ker(A)}P_{Im(G_{\Phi}(\w))}x\|=\|x\|=1$ which is a contradiction. Therefore,
$Ker(A)\cap Im(G_{\Phi}(\w))=\{0\}$ for all $\w\in[-\frac1{2}, \frac1{2}]^d\setminus Z$ and this gives that $(1)$
in Theorem \ref{thm-frames} is satisfied. Having $Ker(A)\cap Im(G_{\Phi}(\w))=\{0\}$ for a.e.
$\w\in[-\frac1{2}, \frac1{2}]^d$, condition $(2)$ of Theorem \ref{thm-frames} can be obtained with similar arguments as in the
first part of this proof.

\end{proof}

\medskip

Now, we consider the case of Riesz bases. We obtain necessary and sufficient conditions on $A$ in order to preserve Riesz  bases of translates.
Different from the frame case,
the conditions on $A$ do not depend on the shift invariant space.

First, observe that, if  $ \Phi=\{\phi_1,\cdots, \phi_m\}\subseteq L^2(\R^d)$ is such that $E(\Phi)$ is an Riesz (orthonormal)
basis for $V=S(\Phi)$, by Proposition \ref{gramiano-frame}, $G_{\Phi}(\w)$ is an $m\times m$ invertible  matrix for a.e.
$\w\in[-\frac1{2}, \frac1{2}]^d$. Thus, using \eqref{length-gramian} we have that $\ell(V)=m$.
Therefore, in order to preserve  Riesz (orthonormal) bases, we need to consider squared matrices $A\in\C^{m\times m}$.

\begin{proposition}\label{thm-for-riesz}
Let $ \Phi=\{\phi_1,\cdots, \phi_m\}\subseteq L^2(\R^d)$ such that $E(\Phi)$ is a Riesz basis  for $V=S(\Phi)$.
Let $A\in\C^{m\times m}$ be a matrix and consider $\Psi=\{\psi_1,\cdots, \psi_m\}$ with
$\Psi=A\Phi$.
Then, $E(\Psi)$ is a Riesz basis for $V$ if and only if $A$ is an invertible matrix.
\end{proposition}

\begin{proof}
Let $0<\alpha\leq \beta$ be the Riesz bounds for $E(\Phi)$.
Suppose first that $A$ is invertible.
Then, for a.e. $\w\in [-\frac1{2}, \frac1{2}]^d,$
\begin{align*}
\|G_{\Psi}(\w)\|&= \|AG_{\Phi}(\w)A^*\|\le \|A\|^2\|G_{\Phi}(\w)\|\le \|A\|^2\beta,\\
\|(G_{\Psi}(\w))^{-1}\|&\le \|A^{-1}\|^2\|(G_{\Phi}(\w))^{-1}\|\le \|A^{-1}\|^2\frac{1}{\alpha}.
\end{align*}
Therefore, by Proposition \ref{gramiano-frame}, $E(\Psi)$ is a Riesz basis for $V$.

Conversely, if $E(\Psi)$ is a Riesz basis for $V,$ it follows from Proposition \ref{gramiano-frame} that
$G_{\Psi}(\w)$ is invertible for a.e. $\w\in[-\frac1{2}, \frac1{2}]^d.$ Since $G_{\Phi}(\w)$ is invertible
 as well for a.e. $\w\in[-\frac1{2}, \frac1{2}]^d$
we have that $A$ is invertible.
\end{proof}

\begin{remark}
The above result shows that every invertible matrix preserves Riesz bases of translates.
On the other side, it is known that the set $\{A\in\C^{m\times m}\colon \det(A)=0\}$ has Lebesgue measure zero.
Thus, we have that almost every matrix (exactly those that are invertible) preserves Riesz bases of translates.
We can then connect this with Bownik and Kaiblinger's result as follows. If in addition to the hypothesis of Theorem \ref{thm-bownik}
we ask the Riesz basis condition on $E(\Phi),$ we have that the set $\mathcal{R}$ is equal to $\{A\in\C^{m\times m}\colon \det(A)\neq0\}.$
\end{remark}

It is worth to mention that Proposition \ref{thm-for-riesz} can be proven as a corollary of Theorem \ref{thm-frames} or independently from this result using Ostrowski's
Theorem \cite[Theorem 4.5.9]{HJ90}.
Applying this result to
$G=G_{\Phi}(\w)$ for a.e.  $\w\in [-\frac1{2}, \frac1{2}]^d$, uniform bounds for the eigenvalues of $AG_{\Phi}(\w)A^*$ can be found.

For frames, in the special case when the initial set of generators has exactly $\ell(V)$ elements, we have that
every invertible matrix yields a set of generators that is a frame for $V.$
This is stated in the next result and its proof is analogous to the proof of Proposition \ref{thm-for-riesz}. It can be also viewed as a corollary of Theorem \ref{thm-Radu}.

\begin{theorem}
Let $ \Phi=\{\phi_1,\cdots, \phi_{\ell}\}\subseteq L^2(\R^d)$ such that $E(\Phi)$ is a frame for $V=S(\Phi)$ and suppose that $\ell(V)=\ell$.
Let $A\in\C^{\ell\times \ell}$ be a matrix and consider $\Psi=\{\psi_1,\cdots, \psi_{\ell}\}$ where $\Psi=A\Phi$.
Then, $E(\Psi)$ is a frame for $V$ if and only if $A$ is an invertible matrix.
\end{theorem}

Finally,  in case of orthonormal bases, we have:

\begin{proposition}\label{thm-for-bons}
Let $ \Phi=\{\phi_1,\cdots, \phi_m\}\subseteq L^2(\R^d)$ such that $E(\Phi)$ is an orthonormal basis for $V=S(\Phi)$ and let $A\in\C^{m\times m}$ be a matrix. Consider $\Psi=A\Phi$. Then,
$E(\Psi)$ is an orthonormal basis for $V$
if and only if $A$ is an unitary matrix.
\end{proposition}

\begin{proof}
Note that if $A=\{a_{ij}\}_{i,j}$ then,
$$
\langle T_k\psi_i, T_{k'}\psi_{i'}\rangle=\sum_{j,j'}^ma_{ij}\overline{a_{i'j'}}
\langle T_k\phi_j, T_{k'}\phi_{j'}\rangle=(AA^*)_{ii'}\delta(k-k')
$$
and from here it follows that $E(\Psi)$ is an orthonormal set if and only if $A$ is unitary.

For the completeness of $E(\Psi)$ on $V$ we use that, since $A$ is unitary and $\Psi=A\Phi,$ we can write $\Phi=A^* \Psi.$ Then
$
T_k\phi_j=\sum_{i=1}^m \overline{a_{ij}} T_k\psi_i, \quad k\in\Z,\, j=1, \ldots, m
$
and the result follows.
\end{proof}

 Theorem \ref{thm-for-riesz} shows that almost every square matrix maps  Riesz bases generators in Riesz bases generators.
For the case of frames it might happen that
condition (2) of Theorem \ref{thm-frames} is not satisfied for any matrix $A.$
That is exactly what we show in the following example on which we present a finitely
generated SIS for which any linear combination of its generators
yields to a minimal set of generators that is not a frame of translates.

\begin{example}
Let $\phi_1, \phi_2\in L^2(\R^2)$ defined by
$$\widehat{\phi_1}(\w_1,\w_2)=-\sin(2\pi\w_1)\chi_{[-\frac1{2}, \frac1{2}]^2}(\w_1,\w_2)$$
and
$$\widehat{\phi_2}(\w_1,\w_2)=e^{2\pi i \w_2}\cos(2\pi\w_1)\chi_{[-\frac1{2}, \frac1{2}]^2}(\w_1,\w_2).$$

Consider the shift invariant space generated by $\Phi=\{\phi_1,\phi_2\}$, $V=S(\phi_1,\phi_2)$.
We will see that $E(\phi_1,\phi_2)$
is a frame for $V$, that $V$ is a principal SIS and that $E(A\Phi)$ is not a frame for $V$  for any matrix $A\in \C^{1\times2}$.

We first compute $G_{\Phi}(\w_1,\w_2)$. Let $(\w_1,\w_2)\in [-\frac1{2}, \frac1{2}]^2$. Then, we have
$$G_{\Phi}(\w_1,\w_2)=
\Big(
\begin{array}{cc}
 \sin^2(2\pi\w_1)& -e^{-2\pi i \w_2}\sin(2\pi\w_1)\cos(2\pi\w_1)\\
-e^{2\pi i \w_2}\sin(2\pi\w_1)\cos(2\pi\w_1) & \cos^2(2\pi\w_1)
\end{array}
\Big).
$$
Note that $G_{\Phi}(\w_1,\w_2)=G_{\Phi}^2(\w_1,\w_2)$ for all
$(\w_1,\w_2)\in [-\frac1{2}, \frac1{2}]^2$. Thus, by Proposition \ref{gramiano-frame}, $E(\phi_1,\phi_2)$
is a frame for $V.$ Further, it can be seen that $\text{rk}(G_{\Phi}(\w_1, \w_2))=1$
for $a.e.$ $(\w_1,\w_2)\in [-\frac1{2}, \frac1{2}]^2$
and then $V$ has length 1. Moreover, $V$ is the Paley Wiener space
$PW=\{f\in L^2(\R^2): \,supp(\widehat{f})\subseteq  [-\frac1{2}, \frac1{2}]^2\}$.

Let $A\in \C^{1\times2}$. Without loss of generality we can suppose that $A=(a_1 \,\,a_2)$ with
$|a_1|^2+|a_2|^2=1$.
Then, $A$ can be written as $A=(\cos(\theta)e^{2\pi i \beta}\,\, \,\,\sin(\theta)e^{2\pi i \beta'})$ for
$\theta\in [0,\frac{\pi}{2}]$ and $\beta, \beta'\in \R$.

Therefore, the Gramian associated to $A\Phi$ is
\begin{align*}
AG_{\Phi}(\w_1,\w_2)A^*&=
 \sin^2(2\pi\w_1) \cos^2(\theta) + \cos^2(2\pi\w_1) \sin^2(\theta)\\
 &-2\cos(2\pi(w_2-\beta'+\beta))\sin(2\pi\w_1)\cos(2\pi\w_1)\sin(\theta)\cos(\theta).
 \end{align*}

Observe that for each $\theta, \beta$ and $\beta'$ fixed,
$$
AG_{\Phi}(\w_1,\w_2)A^*\neq 0\quad  \mbox {for a.e. }  (\w_1,\w_2)\in [-\frac1{2}, \frac1{2}]^2.
$$
In particular,
$$
\text{rk}(G_{\Phi}(\w_1,\w_2))=\text{rk}(AG_{\Phi}(\w_1,\w_2)A^*) \quad  \mbox {for a.e. }  (\w_1,\w_2)\in [-\frac1{2}, \frac1{2}]^2
$$
and then, every matrix $A$ preserves generators.

Let $\tilde{\w}_2\in[-\frac1{2}, \frac1{2}]$ such that $\tilde{\w}_2=\beta'-\beta+k$ for some $k\in\Z$.
Then,
\begin{align*}
AG_{\Phi}(\w_1,\tilde{\w}_2)A^*&=
 \sin^2(2\pi\w_1) \cos^2(\theta) + \cos^2(2\pi\w_1) \sin^2(\theta)\\
 &-2\sin(2\pi\w_1)\cos(2\pi\w_1)\sin(\theta)\cos(\theta)\\
 &=\sin^2(2\pi\w_1-\theta).
 \end{align*}
Now, taking  $\tilde{\w}_1=\frac{\theta}{2\pi},$ we get
$AG_{\Phi}(\tilde{\w}_1,\tilde{\w}_2)A^*=0$.
Then, since the Gramian associated to $A\Phi$ is a continuos function with a zero, condition $(b)$ in item $(3)$
of Proposition \ref{gramiano-frame} can never be fullfiled. Thus,
$E(A\Phi)$ can not be a  frame for $V$.

\vspace{-0.5cm}
\hfill
$\blacksquare$
\end{example}



{\bf Acknowledgments.} The authors thank  J. Antezana and P. Massey  for  fruitful conversations. We also thank J. Antezana
for pointing  out to us  the result in Remark 2.10 from \cite{ACRS05}, and R. Balan for his suggestions
that gave rise to Theorem \ref{thm-Radu}.
Finally, we thank the anonymous referee for her/his comments that helped to improve the manuscript.

\end{document}